\documentclass[11pt]{article}
\usepackage{amsfonts, amsmath, amsthm, amssymb, latexsym, amscd, color}
\usepackage{tabularx,ragged2e,booktabs,caption}
\usepackage{rotating}
\usepackage{mathrsfs,calrsfs}
\usepackage{picinpar}
\usepackage{textcomp}
\usepackage{graphicx,tikz}
\newcommand{\comment}[1]{}
\newif\ifpdf
\ifpdf \else \fi \textwidth = 6.5 in \textheight = 9 in
\newtheorem{thm}{Theorem}[section]

\newtheorem{observation}[thm]{Observation}
\newtheorem{proposition}[thm]{Proposition}

\newtheorem{corollary}[thm]{Corollary}
\newtheorem{lemma}[thm]{Lemma}
\newtheorem{definition}[thm]{Definition}
\newtheorem{theorem}[thm]{Theorem}

\newtheorem{remark}[thm]{Remark}

\begin{document}

\title{Absence of $1$-Nearly Platonic Graphs}
\author{{ Mahdi Reza Khorsandi and Seyed Reza Musawi}
\\{ \small Faculty of Mathematical Sciences, Shahrood University of Technology,}\\
{\small P.O. Box 36199-95161, Shahrood, Iran.}\\ 
{\small khorsandi@shahroodut.ac.ir  and r\_ musawi@shahroodut.ac.ir  }}
\date{}
\maketitle



\begin{abstract}
A $t$-nearly platonic graph is a finite, connected, regular, simple and planar graph in which all but exactly $t$ numbers of its faces have the same length. It is proved that there is no 2-connected $1$-nearly platonic graph. In this paper, we prove that there is no $1$-nearly platonic graph. 
\textbf{Mathematics classification 2010: } 05C07, 05C10.
\end{abstract}

\section{Introduction}
Throughout this paper, all graphs we consider are finite, simple, connected, planar, undirected and non-trivial graph. Suppose that $G=(V,E)$ is a graph with the vertex set $V$ and the edge set $E$. We recall some of the essential concepts, for more details and other terminologies see \cite{Diestel, Bondy-Murty,West-2001}.

A graph is said to be planar, or embeddable in the plane, if it can be drawn in the plane such that each common point  of two edges is a vertex. This drawing of a planar graph $G$ is called a planar embedding of $G$ and can itself be regarded as a graph isomorphic to $G$. Sometimes, we call a planar embedding of a graph as \textit{plane graph}. By this definition, it is clear that we need some matters of the topology of the plane. Immediately, after deleting the points of a plane graph from the plane, we have some maximal open sets (=regions) of points in the plane that is called as faces of the plane graph. There exist exactly one unbounded region that we call it as \textit{outerface} of the plane graph and other faces (if there exist some bounded regions) is called as \textit{internal face}. 
The frontiers of each region is called as the boundary of the corresponding face. The boundary of a face is the set of the points corresponding to some vertices and some edges. In the graph-theoretic language, the boundary of a face is a closed walk. 
A face is said to be incident with the vertices and edges in its boundary, and two faces are adjacent if their boundaries have an edge in common. We denote the boundary of a face $F$ by $\partial(F)$. An \textit{outerplanar graph} is a planar  graph that its outerface is incident with all vertices.

\begin{lemma}\label{rem} \cite[Proposition 6.1.20]{West-2001} Every simple outerplanar graph with at least four vertices has at least two nonadjacent vertices of degree at most 2.
\end{lemma}

\begin{proposition}\cite[Proposition 10.5]{Bondy-Murty}
Let $G$ be a planar graph, and let $f$ be a face in some planar embedding of $G$. Then $G$ admits a planar embedding whose outerface has the same boundary as $f$ .
\end{proposition}

A graph $G$ is called $k$-regular when the degrees of all vertices are equal to $k$. A regular graph is one that is $k$-regular for some $k$. Let $G=(V,E)$ be a graph with the vertex set $V$ and the edge set $E$. We will show the number of vertices of $G$ by $n=|V|$, the number of edges of $G$ by $m=|E|$ and the number of faces of $G$ by $f$. The \textit{Euler's formula} states that if $G$ is a connected planar graph, then: \[m-n=f-2\]

\begin{definition}
Let $G$ be a graph and $S\subseteq V(G)$. Then $\langle S\rangle$,  the induced subgraph by $S$, denotes the graph on $S$ whose edges are precisely the edges of $G$ with both ends in $S$. Also, $G-S$ is obtained from $G$ by deleting all the vertices in $S$ and their incident edges. If $S=\{x\}$ is a singleton, then we write $G-x$ rather than $G-\{x\}$.
\end{definition}

The \textit{length} of a face in a plane graph $G$ is the total length of the closed walk(s) in $G$ bounding the face. A cut-edge belongs to the boundary of only one face, and it contributes twice to its length (see \cite[Example 6.1.12]{West-2001}). $G$ is called $k$-connected (for $k\in \mathbb{N}$) if $|V(G)|>k$ and $G-X$ is connected for every set $X\subseteq V(G)$ with $|X|<k$.
If $|V(G)|>1$ and $G-F$ is connected for every set $F\subseteq E(G)$  of fewer than $\ell$ edges, then $G$ is called $\ell$-edge-connected. 

\begin{proposition}\label{2e} \cite[Proposition 6.1.13]{West-2001} If $\ell(F_i)$ denotes the length of the face $F_i$ in a plane graph $G$, then $2m =\underset{i}{\sum} \ell(F_i)$ . 
\end{proposition}

\begin{theorem} \cite[Page 312] {West-2001}
 In a 2-edge-connected plane graph, all facial boundaries are cycles and each edge lies in the boundary of two faces. 
\end{theorem}

\begin{figwindow}[2,r,{
\includegraphics[scale=.4,angle=90]{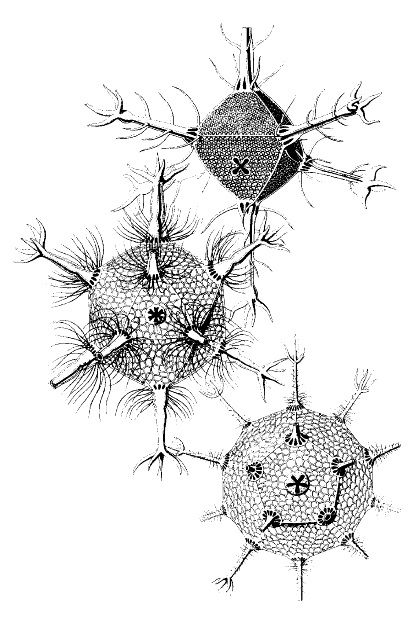}
},{3 radiolarians}\label{radiolarians}]
Platonic solids are a well-known five-membered family of 3-dimensional polyhedra. There is no reliable information about their date of birth, and different opinions have been taken \cite{Atiyah-Sutcliffe,Lloyd}. However, they are attractive for mathematicians and others, in terms of some symmetries that they have. In the last two centuries, many of authors have paid attention to the polyhedra and they have extended it to convex and concave polytopes in the different dimensions \cite{Grunbaum}. Older scientists, such as Kepler and Plato, describe the properties of Platonic solids that we know that  they are not right, but nowadays, with the new advances in a variety of sciences such as physics, chemistry, biology, etc., we observe some applications of polyhedra and especially platonic solids (see \cite{Atiyah-Sutcliffe} and \cite[Figure \ref{radiolarians}]{Weyl-1952}). 
\end{figwindow}

But what matters from the combinatorial point of view is that a convex polyhedron can be embedded on a sphere, and then we can map it on a plane so that the images of lines on the sphere do not cut each other in the plane. In this way, we have corresponded a polyhedron on the sphere with a planar graph in the plane. 
Steinitz's theorem (see \cite[p. 235]{Grunbaum}) states that a graph $G$ with at least four vertices is the network of vertices and edges of a convex polyhedron if and only if $G$ is planar and $3$-connected.  
In 1967, Grunbaum considered 3-regular and connected planar graphs and he got some results. For example, for a $3$-regular connected planar graph and $k\in\{2,3,4,5\}$,  it is proved that if the length of all faces but $t$ faces is divisible by $k$ then $t\ge2$ and if $t=2$ then two exceptional faces have not a common vertex \cite{Grunbaum}. In 1968, in his Ph.D thesis, Malkevitch proved the same results for 4 and 5-regular 3-connected planar graphs \cite{Malkevitch-2}.  
Several papers are devoted to the study of this topic, but all of them have considered the planar graphs such that the lengths of all faces but some exceptional faces are a multiple of $k$ and $k\in \{2,3,4,5\}$ (see \cite{Crowe,Jendrol-Jucovic,Jendrol,Hornak-Jendrol}).

Recently, Keith et al. in \cite{Keith-Froncek-Kreher-1}, defined a $t$-nearly platonic graph to be a finite $k$-regular simple planar graph in which all faces, with the exception of $t$ numbers of the faces, have the same length.  They proved that there is no 1-nearly platonic graph. 
  However, their proof is only valid for 2-connected graphs (see \cite{Keith-Froncek-Kreher-2}). In this paper, we prove that there is no 1-nearly platonic graph. This is a strength of the Theorem  1 in \cite{Keith-Froncek-Kreher-2} and  completes the proof of the Theorem 6 in \cite{Keith-Froncek-Kreher-1} about 1-nearly platonic graph.
 
%
%
%

%
%
%
%
%




\section{Absence of 1-nearly platonic graphs}

\begin{definition}
A $k$-regular simple connected planar graph is a $(k;d_1^{f_1}d_2^{f_2}\cdots d_{\ell}^{f_{\ell}})$-graph if it has $f_i$ faces of degree $d_i$, $i=1,2,\cdots,\ell$, where $f=f_1+f_2+\cdots+f_{\ell}$.
\end{definition}
\begin{definition}
A $t$-nearly platonic graph is a finite, connected, regular, simple and planar graph in which all but exactly $t$ numbers of its faces have the same length. For simplicity, we show a $t$-nearly platonic graph by $t$-NPG  or $k$-regular $t$-NPG to emphasize the valency of graph. 
\end{definition}

If the graph $G$ is a $k$-regular $t$-NPG, then by planarity of $G$, it is obvious that $k\in\{1,2,3,4,5\}$. For $k=1$ we have $G=K_2$ with one face and for $k=2$ we have $G=C_n$ with two faces with the same lengths, so we have nothing to say. Hence, from now on we assume that $k\in\{3,4,5\}$ and so $n\ge4$, $m\ge6$ and $f\ge4$. 
A $k$-regular $1$-NPG can be written as $(k;d_1^{f-1}d_2^{1})$-graph such that $d_1,d_2\ge3$, $d_1\ne d_2$ and the unique face with the length $d_2$ is called the exceptional face of the $1$-NPG. Keith et al. \cite{Keith-Froncek-Kreher-1} proved that there is no 2-connected $(k;d_1^{f-1}d_2^{1})$-graph. In other words:

\begin{theorem}\label{th.2.3}\cite[Theorem 1]{Keith-Froncek-Kreher-2}
There is no finite, planar, 2-connected regular graph that has all but one face of one degree (length) and a single face of a different degree (length).
\end{theorem}

\begin{definition} In a planar graph, we call a vertex or an edge as an e-vertex or an e-edge, respectively, if it lies on the boundary of the outerface, otherwise, we call them as an i-vertex or an i-edges. 
\end{definition}

\begin{observation}\label{rem111}
\begin{itemize}
  \setlength\itemsep{.01em}
\item[(i)]
In each planar graph, all edges passing through an i-vertex are i-edges.
\item[(i)]
In a 2-connected planar graph, precisely two edges passing through an e-vertex are e-edges and others are i-edges.
\end{itemize}
\end{observation}

\begin{figwindow}[1,r,{
\unitlength1cm
 \begin{tikzpicture}[scale=1.2]
 \draw[white] (-1.8,1)--(2.75,1);
 \draw (-1,0.3)--(0,0)--(0,2)--(-1,1.7);
\draw (0,0) arc (-90:90:2cm and 1cm);
\draw (1,1)--(2,1);
\filldraw (0,0) circle(2pt);\node at (0,-.3){u};
\filldraw (0,2) circle(2pt);\node at (0,2.3){v};
\filldraw (2,1) circle(2pt);\node at (2.3,1){x};
\filldraw (1,1) circle(2pt);\node at (1.2,1.2){y};
\draw [dashed](1,1) arc (0:360:.35cm and .5cm);
\end{tikzpicture}
},{$x$ root an inflorescence.}
\label{inflorescence}]
In \cite[Figure 1]{Keith-Froncek-Kreher-1}, a special instance in planar graphs is called an inflorescence. 
We say that the vertex $x$ root an inflorescence if it must be adjacent to a vertex $y$ within a face not on the boundary of it, since the boundary vertices have already known their neighbours. $y$ must also be adjacent to some other vertices within this face. But this makes the edge $xy$ a cut-edge, and makes the vertices $x$ and $y$ cut-vertices. This is illustrated in Figure \ref{inflorescence}. We will use this result many times in drawing the planar graph. Indeed, if a vertex root an inflorescence, then the graph has a cut-vertex and so it is not a 2-connected graph.
\end{figwindow}

\begin{theorem}\label{i-vertex}
Suppose that $G$ is a $2$-connected planar graph such that all vertices are of degree $3$ except only one vertex on the boundary of its outerface that has the degree $2$ and all internal faces have the same length. Then we have:
\begin{itemize}
  \setlength\itemsep{.01em}
\item[(i)]
The length of each internal face is 3, 4 or 5.
\item[(ii)]
The boundary of the each face of $G$ is an induced cycle.
\item[(iii)]
Each e-vertex has exactly two neighbours on the outerface.
\item[(iv)]
If two internal faces have a common vertex, then they have exactly one edge and two adjacent vertices in common.
\end{itemize}
\end{theorem}

\begin{proof}
\textbf{(i)}
Let $n$ be the number of vertices, $m$ be the number of edges, $f$ be the number of faces, $l$ the length of external face and $d$ be the length of each internal face of $G$.
By enumerating the edges of $G$ in two ways, we have $2m=3n-2$, $2m=l+(f-1)d$ and by Euler's formula, $f-1=m-n+1=(n+1)/2$. Therefore, we have $3n-2=l+d(n+1)/2$ and so $d=6-(2l+10)/(n+1)<6$, as desired. \\
\textbf{(ii)} First, we consider the internal faces of $G$. By part (i), the length of each internal face is 3, 4 or 5. \\
\textbf{Case 1:} The lengths of internal faces are 3. Obviously, the boundary of each internal face is an induced cycle $C_3$.\\
\textbf{Case 2:} The lengths of internal faces are 4. The boundary of each internal face is a cycle $C_4$. Let $yuvwy$ be the boundary of an internal face. If $yuvwy$ is not an induced cycle, then there exist exactly one chord $yv$ or $uw$. By symmetry, we consider the chord $yv$. Now, either the vertex $u$ belongs to the interior of the triangle $ywvy$ or the vertex $w$ belongs to the interior of the triangle $yuvy$. Again by symmetry, suppose that $u$ belongs to the interior of triangle $ywvy$. 
Since $u$ is an i-vertex, $u$ root an inflorescence, a contradiction to 2-connectivity of the graph.  Therefore, $yuvwy$ is an induced cycle, as desired.\\
\textbf{Case 3:} The lengths of internal faces are 5. The boundary of each internal face is a cycle $C_5$. Let $xyuvwx$ be the boundary of an internal face. If $xyuvwx$ is not an induced cycle, then there exist two chords. By symmetry, assume that  $yv$ is a chord. Now, we consider two subcases.\\
\textbf{Subcase 1.} The vertex $u$ belongs to the interior of the square $xyvwx$. In this subcase, $u$ root an inflorescence, a contradiction.\\
\textbf{Subcase 2.} The vertices $x$ and $w$ belongs to the interior of the square $yuvy$. Since $y$  and $v$ is already of degree 3, the path $xyvw$ is the part of the boundary of a pentagonal face and so the vertices $x$ and $w$ have a common neighbour, say $z$. But the vertex $z$ root an inflorescence in the triangle $xwzx$, a contradiction.
 Therefore, $yuvwy$ is an induced cycle.  
 In this 3 cases we show that the boundaries of all internal faces are induced cycles, as desired. 
 
 Now, we consider the boundary of the outerface. Assume that the cycle $x_1x_2x_3\cdots x_lx_1$ is the boundary of the outerface such that $\deg(x_l)=2$ and the degrees of all other vertices of $G$ are $3$.  
By the contrary, suppose that $\langle\{x_1,x_2,\cdots ,x_l\}\rangle\ne C_l$. Hence there exist a chord $x_ix_j$ such that $1\le i< j\le l$ and $x_ix_j\notin\{x_ix_{i+1}:i=1,2,\cdots,l-1\}\cup\{x_1x_l\}$. Note that $j\ne l$. 
Otherwise, $x_i\in N(x_l)$ and $x_ix_l\notin\{x_{l-1}x_{l},x_1x_l\}$ and so $i\notin\{1,l-1\}$. This implies that $\deg(x_l)\ge3$, a contradiction. Also, we have $j\ne i+2$. Otherwise, we consider the cycle $x_ix_{i+1}x_{i+2}x_i$ and we see that the vertex $x_{i+1}$ root an inflorescence, a contradiction. Therefore, we have  $1\le i\le j-3\le l-4$.

Let $S$ be the set of all vertices lying on the cycle $x_ix_{i+1}x_{i+2}\cdots x_jx_i$ and its interior. The subgraph $H=\langle S \rangle$ is a $2$-connected planar graph such that $\deg(x_i)=\deg(x_j)=2$, the degrees of all other vertices of $H$ are $3$, the length of the outerface is at least $4$ and the internal faces have the same length.

We consider two copies of $H$ and we construct the new graph $H'$ by matching the vertices of degree $2$ from one copy to their respective vertices from the other copy of $H$. The graph $H'$ is a $3$-regular $2$-connected planar graph such that all internal faces have the same length equal to 3, 4 or 5, while the length of the outerface of $H'$ is at least $6$, therefore, $H'$ is a $2$-connected $1$-NPG, that contradicts the Theorem \ref{th.2.3}. Hence, the boundary of the outerface is an induced cycle.\\
\textbf{(iii)} Since the outerface is a cycle, each e-vertex has at least two neighbours on this cycle. If the third neighbour of an e-vertex lies on the outerface, then the cycle has a chord and so it is not an induced cycle, a contradiction with the part \textbf{(ii)}.
\\
\textbf{(iv)} By part (i), the length of each internal face is 3, 4 or 5.\\
\textbf{Case 1:} The lengths of internal faces are 3. Assume that two triangles $uu_1u_2u$ and $uv_1v_2u$ are two different internal faces and $u$ lies on both of them. If 
$\{u_1,u_2\}\cap\{v_1,v_2\}=\emptyset$, then $\deg(u)\ge 4$, a contradiction, and so $\{u_1,u_2\}\cap\{v_1,v_2\}\ne \emptyset$. We set $u_1=v_1$, that is, two triangles have an edge and two adjacent vertices in common. We have $u_2\ne v_2$, otherwise, two triangles are the same, a contradiction.\\
\textbf{Case 2:}  The lengths of internal faces are 4. Assume that two squares $uu_1u_2u_3u$ and $uv_1v_2v_3u$ are the boundaries of two internal faces and $u$ lies on both of them. If $\{u_1,u_3\}\cap\{v_1,v_3\}=\emptyset$, then $\deg(u)\ge 4$, a contradiction. We set $u_1=v_1$, that is, two squares have an edge and two adjacent vertices in common. We show that $\{u_2,u_3\}\cap \{v_2,v_3\}=\emptyset$. Since $u_2\notin N(u)=\{u_1,u_3,v_3\}$, we have $u_2\ne v_3$. 
If $u_2=v_2$, then the vertex $u_1$ lies inside the square $uu_3u_2v_3u$. It is an $i$=vertex with $\deg(u_1)=3$. In this case, either the graph has a chord $u_1u_3$ or $u_1v_3$, or $u_1$ root an inflorescence, a contradiction.
Similarly $u_3\notin\{v_2,v_3\}$.\\
\textbf{Case 3:} The lengths of internal faces are 5. Assume that two pentagons $uu_1u_2u_3u_4u$ and $uy_1y_2y_3y_4u$ are the boundaries of two internal faces and $u$ lies on both of them. If $\{u_1,u_4\}\cap\{v_1,v_4\}=\emptyset$, then $\deg(u)\ge 4$, a contradiction. We set $u_1=y_1$, that is, two pentagons have an edge and two adjacent vertices in common. We show that $\{u_2,u_3,u_4\}\cap\{v_2,v_3,v_4\}=\emptyset$. Since $u_2\in N(u_1)=\{u,u_2,v_2\}$, we have $u_2\notin\{v_3,v_4\}$. If $u_2=v_2$, then the vertex $u_1$ lies inside the hexagon $uu_4u_3u_2v_3v_4u$. It is an $i$=vertex with $\deg(u_1)=3$. In this case, either the graph has a chord or $u_1$ root an inflorescence, a contradiction. Similarly $u_4\notin\{v_2,v_3,v_4\}$ and $\{u_2,u_3,u_4\}\cap\{v_2,v_4\}=\emptyset$. Finally, if  $u_3=v_3$, then $\deg(u_3)=4$, a contradiction. 
\end{proof}

The parts (ii) and (iii) of the Theorem \ref{i-vertex} play an important role in the construction of graphs in the following theorem.

\begin{theorem}\label{th3.8}
There is no $2$-connected planar graph such that:
\begin{itemize}
  \setlength\itemsep{.01em}
\item[(i)]
All vertices are of degree $3\le k\le5$ except only one vertex on outerface that has the degree $k_0$ and $2\le k_0\le k-1$.
\item[(ii)]
 All internal faces have the same length.
 \end{itemize}
\end{theorem}

\begin{proof}
Assume to the contrary that there exists a $2$-connected planar graph, $G$, such that the length of all internal faces of $G$ is $d$ and the length of the outerface is $\ell$. The graph $G$ has a vertex $x$ on its outerface, with $\deg(x)=k_0$ and the degree of all other vertices is $k$, where $2\le k_0\le k-1$. 
Since $G$ is a $2$-connected graph, the boundary of each face is a cycle, thus the number of vertices and the number of edges lying on the outerface is equal to $\ell$ and also, $d,\ell\ge3$. By Lemma \ref{rem}, $G$ is not an outerplanar graph and so $\ell\le n-1$. \\
We have some relations between the parameters of $G$:
\begin{align}\label{f1}
2m&=(n-1)k+k_0 , \\ \label{f2}
2m&=(f-1)d+\ell
\end{align}
Now, by Euler's formula we have:
\begin{align}\label{f3}
f-1=m-n+1=\frac{1}{2}[(n-1)(k-2)+k_0]
\end{align}
Hence, by (\ref{f1}), (\ref{f2}), and (\ref{f3}) we conclude that:
\begin{align}\label{f4}
(n-1)(2k+2d-dk)=k_0(d-2)+2\ell
\end{align}
Since $k_0(d-2)+2\ell\ge8$, by equality (\ref{f4}), we see $2k+2d-dk>0$, that is, $\frac{2}{d}+\frac{2}{k}>1$ which implies that $k=3, d\in\{3,4,5\}$ or $k\in\{4,5\}, d=3$. In each case, we have $k\ge 3$ and so $n\ge4$. Thus, we have $8$ cases to check:
\\
\textbf{Case 1.} If $k=3, d=3, k_0=2$, then by (\ref{f4}), we have $3n-5=2\ell\le 2(n-1)$ and so $n\le3$, a contradiction.
\begin{figwindow}[1,r,{
\unitlength1cm
 \begin{tikzpicture}[scale=.85]
 \draw (5,0)--(7,1)--(7,4)--(5,5)--(3,4)--(3,1)--(5,0);
 \filldraw (5,5) circle (2pt);\node at (5,5.3){$x$};
    \filldraw (3,4) circle (2pt);\node at (2.7,4){$x_1$};
      \filldraw (7,4) circle (2pt);\node at (7.3,4){$y_1$};
    \draw (7,4)--(5,3)--(3,4);
   \filldraw (5,3) circle (2pt);\node at (5,3.3){$z_1$};   
     \filldraw (3,1) circle (2pt);\node at (2.7,1){$x_2$};
   \filldraw (7,1) circle (2pt);\node at (7.3,1){$y_2$};
\filldraw (5,2) circle (2pt);\node at (5,1.7){$z_2$};     
\draw (5,2)--(5,3);
\draw (3,1)--(5,2)--(7,1);
     \filldraw (5,0) circle (2pt);\node at (5,-.3){$y$};
 \end{tikzpicture}
},{$k=3, d=4, k_0=2$}]
\noindent
\textbf{Case 2.} If $k=3, d=4, k_0=2$, then $x$ has two neighbours. Let $N(x)=\{x_1,y_1\}$. By Theorem \ref{i-vertex}(iii), the edges $xx_1$ and $xy_1$ are e-edges. They lie on an internal face, namely, the square $xx_1z_1y_1x$. If the outerface is incident to $z_1$, then $G=C_4$, a contradiction. 
Hence, the vertex $z_1$ is an i-vertex and the edges $x_1z_1$ and $y_1z_1$ are i-edges.
$x_1z_1$ is belonging to another squares, say  $x_1z_1z_2x_2x_1$. The vertex $z_1$ has all 3 neighbours and so, the path $y_1z_1z_2$ is a part of the new square  $y_1z_1z_2y_2y_1$. 
$x_2$ is the third neighbour of  $x_1$ and $y_2$ is the third neighbour of  $y_1$ and so they are e-vertices. 
We claim that $z_2$ does not lie on the outerface. If $z_2$ lies on the outerface, then the graph has no other vertex, but it has three vertices of degree $2$, a contradiction. Since $z_2$ is an i-vertex with degree $3$, the path $x_2z_2y_2$ is a part of a square, say $x_2z_2y_2yx_2$. Now, $y$ is the third neighbour of $x_2$ and $y_2$. $y$ necessarily lies on the outerface and the graph has no other vertex, but it has two vertices of degree $2$, a contradiction.
\end{figwindow}

\begin{figwindow}[7,r,{
 \begin{tikzpicture}[scale=1.2]
 \draw (4.6,3.7)--(4.3,4.5)--(5,5)--(5.7,4.5)--(5.4,3.7)--(4.6,3.7);
 \filldraw (5,5) circle (2pt);\node at (5,5.3){$x$};
  \filldraw (4.3,4.5) circle (2pt);\node at (4,4.5){$x_1$};
   \filldraw (5.7,4.5) circle (2pt);\node at (6,4.5){$y_1$};
\filldraw (4.6,3.7) circle (2pt);\node at (4.7,3.9){$z_1$};   
\filldraw (5.4,3.7) circle (2pt);\node at (5.3,3.9){$z_2$};   
\draw (4.3,4.5)--(3.5,2.5)--(4.1,2.5)--(4.6,3)--(4.6,3.7);
\draw (5.7,4.5)--(6.5,2.5)--(5.9,2.5)--(5.4,3)--(5.4,3.7);
\filldraw (3.5,2.5) circle (2pt);\node at (3.2,2.5){$x_2$};           
\filldraw (6.5,2.5) circle (2pt);\node at (6.8,2.5){$y_2$};
\filldraw (4.1,2.5) circle (2pt);\node at (4.4,2.5){$z_6$};     
\filldraw (4.6,3) circle (2pt);\node at (4.3,3){$z_3$};
\filldraw (5.4,3) circle (2pt);\node at (5.7,3){$z_4$};
\filldraw (5.9,2.5) circle (2pt);\node at (5.6,2.5){$z_7$};  
\filldraw (5,2.7) circle (2pt);\node at (5,2.9){$z_5$}; 
\filldraw (5,2.3) circle (2pt);\node at (5,2.1){$z_8$};   
\draw (4.6,3)--(5,2.7)--(5.4,3);
\draw (5,2.7)--(5,2.3);
\draw (4.1,2.5)--(4.6,2)--(5,2.3)--(5.4,2)--(5.9,2.5);
\filldraw (4.6,2) circle (2pt);\node at (4.6,2.2){$z_9$}; 
\filldraw (5.4,2) circle (2pt);\node at (5.4,2.2){$z_{10}$};  
\draw (3.5,2.5)--(4.3,0.5)--(4.6,1.3)--(4.6,2);
\draw (6.5,2.5)--(5.7,0.5)--(5.4,1.3)--(5.4,2);
\filldraw (4.6,1.3) circle (2pt);\node at (4.8,1.1){$z_{11}$}; 
\filldraw (5.4,1.3) circle (2pt);\node at (5.3,1.1){$z_{12}$};  
\filldraw (4.3,0.5) circle (2pt);\node at (4,.5){$x_3$}; 
\filldraw (5.7,.5) circle (2pt);\node at (6,.5){$y_3$}; 
\draw (4.6,1.3)--(5.4,1.3);
\draw (4.3,0.5)--(5,0)--(5.7,.5);
\filldraw (5,0) circle (2pt);\node at (5,-.3){$y$}; 
 \end{tikzpicture}
},{$k=3, d=5, k_0=2$}]
\noindent
\textbf{Case 3.} If $k=3, d=5, k_0=2$, then $x$ has two neighbours. Let $N(x)=\{x_1,y_1\}$. By Theorem \ref{i-vertex}(iii), the edges $xx_1$ and $xy_1$ are e-edges. They lie on an internal face, the pentagon $y_1xx_1z_1z_2y_1$.  If $z_1$ be an e-vertex, then $\deg(x_1)=2$, a contradiction. Hence, $z_1$ is an i-vertex and $x_1z_1$ is an i-edge. Similarly, $z_2$ is an i-vertex and $y_1z_2$ is an i-edge. Also, $z_1z_2$ is an i-edge and so it lies on the second pentagonal face $z_1z_2z_4z_5z_3z_1$. By Theorem \ref{i-vertex}(iv), we have $\{z_3,z_4,z_5\}\cap \{x,x_1,y_1\}=\emptyset$. 
$x_1z_1$ lies on the boundary of another pentagon, say $x_1z_1z_3z_6x_2x_1$. By Theorem \ref{i-vertex}(iv), we have $\{x_2,z_3,z_6\}\cap \{x,y_1,z_2,z_4,z_5\}=\emptyset$. 
The i-edge $y_1z_2$ lies on the second pentagon, $y_1z_2z_4z_7y_2y_1$ and $\{y_2,z_7\}\cap \{x,x_1,z_1,z_3,z_5\}=\emptyset$. 
Since $x_2$ and $y_2$ are the third neighbours of $x_2$ and $y_2$, respectively, the edges  $x_1x_2$ and $y_1y_2$ are e-vertices and so $x_2$ and $y_2$ are e-vertices. 
The vertices $z_6$ and $z_7$ are i-vertices. Otherwise, $\deg(x_2)=2$ or $\deg(y_2)=2$, a contradiction.
By Theorem \ref{i-vertex}(iv), $z_6=z_7$ if and only if $x_2=y_2$. If $z_6=z_7$, then By Theorem \ref{i-vertex}(iv), $x_2=y_2$ and so the cycle $xx_1x_2y_1x$ is the boundary of outerface and graph is completed but,  $z_5$ root an inflorescence, a contradiction. 
Consequently, $z_6\ne z_7$ and $x_2\ne y_2$.
Also, $z_6\ne y_2$, otherwise, we have $\deg(y_2)=4$, a contradiction. Similarly, $z_7\ne x_2$. 
The i-edge $z_3z_6$ lies on the second pentagon $z_3z_6z_9z_8z_5z_3$. By Theorem \ref{i-vertex}(iv), $\{x_2\}\cap\{z_8,z_9\}=\emptyset$.  
By Theorem \ref{i-vertex}(iv), $z_7=z_8$ if and only if $z_9=y_2$. In this case, the graph  has a triangular internal face $z_4z_5z_7z_4$, a contradiction and Again, by Theorem \ref{i-vertex}(iv), $z_8=y_2$ if and only if $z_7=z_9$. This case contradicts the planarity of the graph. Therefore, we have $\{y_2,z_7\}\cap\{z_8,z_9\}=\emptyset$.
The i-edge $z_4z_7$ lies on the second pentagon $z_4z_5z_8z_{10}z_7z_4$. By Theorem \ref{i-vertex}(iv), $z_{10}\notin\{y_2,z_9\}$ and $x_2\ne z_{10}$, otherwise, $\deg(x_2)=4$, a contradiction.
 The i-edge $x_2z_6$ lies on the second pentagon $x_2z_6z_9z_{11}x_3x_2$. If $z_{11}=z_{10}$ we have a triangular internal face $z_8z_9z_{10}$, a contradiction. If $z_{11}=y_2$ then we have $\deg(y_2)=4$, a contradiction. If $x_3=z_{10}$ or $x_3=y_2$, then we have $\deg(z_{10})=4$ or $\deg(y_2)=4$, respectively, a contradiction. Thus, $\{x_3,z_{11}\}\cap \{y_2,z_{10}\}=\emptyset$.
   Furthermore, $x_3$ is the third neighbour of $x_2$ and so it is an e-vertex. $z_{11}$ is an i-vertex  and $x_3z_{11}$ is an i-edge. Otherwise, $\deg(x_3)=2$, a contradiction. Since $z_9z_{11}$ is an i-edge, it lies on the second pentagon $z_9z_{11}z_{12}z_{10}z_8z_9$. 
 If $z_{12}=x_3$, then $z_{11}$ root an inflorescence and if $z_{12}=y_2$, then $\deg(y_2)=4$, a contradiction. Thus, $\{z_{12}\}\cap \{x_3,y_2\}=\emptyset$.
  The edge $y_2z_7$ is an i-edge and so it lies on the second pentagon $y_2z_7z_{10}z_{12}y_3y_2$. If $x_3=y_3$, then $\deg(x_3)=4$, a contradiction. Furthermore, since $y_3$ is the third neighbour of $y_2$, it is an e-vertex. The i-edge $x_3z_{11}$ has to lie on the second pentagon $x_3z_{11}z_{12}y_3yx_3$. The vertex $y$ is the third neighbour of $y_3$ and so it is an e-vertex and the cycle $xx_1x_2x_3yy_3y_2y_1x$ is the boundary of outerface of the graph $G$. The graph $G$ is completed while $\deg(y)=2$ or $y$ root an inflorescence, a contradiction.
\end{figwindow}

\begin{figwindow}[4,r,{
\begin{tikzpicture}[scale=.95]
\draw[white] (2.5,2)--(7.5,2);
  \draw (6.65,2)--(6.65,4)--(5,5)--(3.35,4)--(3.35,2)--(5,3)
  --(6.65,2)--(5,1)--(3.35,2)--(6.65,2);
  \draw (3.35,4)--(6.65,4)--(5,3)--(3.35,4);
 \filldraw (5,5) circle (2pt);\node at (5,5.3){$x$};
  \filldraw (3.35,4) circle (2pt);\node at (3,4){$x_1$};
   \filldraw (6.65,4) circle (2pt);\node at (7,4){$y_1$};
\filldraw (5,3) circle (2pt);\node at (5,2.7){$z$};   
\filldraw (3.35,2) circle (2pt);\node at (3,2){$x_2$};           
\filldraw (6.65,2) circle (2pt);\node at (7,2){$y_2$};
\filldraw (5,1) circle (2pt);\node at (5,.7){$y$}; 
\end{tikzpicture}
},{$k=4, d=3, k_0=2$}]
\noindent
\textbf{Case 4.} If $k=4, d=3, k_0=2$, then $x$ has two neighbours. Let $N(x)=\{x_1,y_1\}$. By Theorem \ref{i-vertex}(iii), the edges $xx_1$ and $xy_1$ are e-edges. The second face consisting the edge $xx_1$ is the triangle  $xx_1y_1x$. The edge $x_1y_1$ is an i-edges. Otherwise, $G=C_3$, that has three vertices of degree $2$, a contradiction. The second face consisting the edge $x_1y_1$ is the triangle  $x_1y_1zx_1$. If $z$ be an e-vertex, then $G$ has no other vertex and so $G$ has two vertices of degree $2$, a contradiction. Therefore, $z$ is an i-vertex and consequently the edges $x_1z$ and $y_1z$ are i-edges. The second face consisting the edge $x_1z$ is the triangle  $x_1zx_2x_1$ and the second face consisting the edge $y_1z$ is the triangle  $y_1zy_2y_1$. If $x_2=y_2$, then $\deg(z)=3\ne k$ or $z$ root an inflorescence, a contradiction. Therefore, $x_2$ and $y_2$ are distinct vertices. Since $x_2$ and $y_2$ are the fourth neighbours of $x_1$ and $y_1$, respectively, they are e-vertices. Now, $\deg(z)=4$ and so the second face consisting the i-edge $x_2z$ is the triangle  $x_2zy_2x_2$. 
The edge $x_2y_2$ is an i-edge, otherwise, $\deg(x_2)=\deg(y_2)=3<k$, a contradiction. Let $x_2y_2yx_2$ be the second triangle incident to $x_2y_2$. 
Now, the edges $x_2y$ and $y_2y$ are the fourth edges passing through $x_2$ and $y_2$, respectively, and so they are the second e-edges passing through $x_2$ and $y_2$. That is, $y$ is an e-vertex and the graph $G$ is completed while $\deg(y)=2$ or $y$ root an inflorescence, a contradiction.
\end{figwindow} 
\noindent
\textbf{Case 5.} If $k=4, d=3, k_0=3$, then $x$ is the only vertex by an odd degree while other vertices have an even degree, it is impossible.

\begin{figwindow}[2,r,{
\begin{tikzpicture}[scale=1]
\draw (0,0)--(3,2)--(0,6)--(-3,2)--(0,0);
\draw (0,0)--(0,1)--(3,2)--(.75,3.7)--(0,4.3)--(-.75,3.7)--(-3,2)--(0,1);
\draw (.75,5)--(.75,2)--(0,1)--(-.75,2)--(-.75,5);
\draw (0,2.85)--(0,4.3)--(.75,5)--(-.75,5)--(0,4.3);
\draw (-3,2)--(3,2);
\draw (.75,3.7)--(-.75,2);\draw (-.75,3.7)--(.75,2);
 \filldraw (0,6) circle (2pt);\node at (0,6.3){$x$};
\filldraw (-.75,5) circle (2pt);\node at (-1.1,5){$x_1$};
\filldraw (-3,2) circle (2pt);\node at (-3.3,2){$x_2$};
\filldraw (.75,5) circle (2pt);\node at (1.1,5){$y_1$};
\filldraw (3,2) circle (2pt);\node at (3.3,2){$y_2$};
\filldraw (0,0) circle (2pt);\node at (0,-.3){$y$};
\filldraw (0,4.3) circle (2pt);\node at (0,4.6){$z_1$};  
\filldraw (-.75,3.7) circle (2pt);\node at (-1,3.8){$z_2$};  
\filldraw (.75,3.7) circle (2pt);\node at (1,3.8){$z_3$}; 
\filldraw (0,2.85) circle (2pt);\node at (0,2.5){$z_4$};
\filldraw (-.75,2) circle (2pt);\node at (-.75,1.7){$z_5$};  
\filldraw (.75,2) circle (2pt);\node at (.75,1.7){$z_6$};  
\filldraw (0,1) circle (2pt);\node at (0,1.3){$z_7$}; 
\end{tikzpicture}
},{$k=5, d=3, k_0=2$}]
\noindent
\textbf{Case 6.} If $k=5, d=3, k_0=2$, then $x$ has two neighbours. Let $N(x)=\{x_1,y_1\}$. By Theorem \ref{i-vertex}(iii), the edges $xx_1$ and $xy_1$ are e-edges. The second face consisting the edge $xx_1$ is the triangle  $xx_1y_1x$. The edge $x_1y_1$ is an i-edges. Otherwise, $G=C_3$, that has three vertices of degree $2$, a contradiction. 
The second face consisting the edge $x_1y_1$ is the triangle  $x_1y_1z_1x_1$. If $z_1$ be an e-vertices, then $G$ has no other vertices and so $G$ has two vertices of degree $2$, a contradiction. Therefore, $z_1$ is an i-vertex and consequently the edges $x_1z_1$ and $y_1z_1$ are i-edges. The second face consisting the edge $x_1z_1$ is the triangle  $x_1z_1z_2x_1$ and the second face consisting the edge $y_1z_1$ is the triangle  $y_1z_1z_3y_1$. Note that if $z_2=z_3$, then $\deg(z_1)=3$ which is a contradiction. Hence, $z_2\ne z_3$.  
We call the fifth neighbour of $z_1$ as $z_4$ and so the i-edge $z_1z_4$ has to lie on two triangle $z_1z_2z_4z_1$ and $z_1z_3z_4z_1$. 
The vertices  $z_2$ and $z_3$ are $i$-vertices, otherwise, we have $\deg(x_1)=4<k$ or $\deg(y_1)=4<k$, respectively, a contradiction.
The i-edge $z_2z_4$ lies on the second triangle $z_2z_4z_5z_2$. If $z_5=x_1$, then $\deg(z_2)=3$ and if $z_5=z_3$, then $\deg(z_4)=3$ and if $z_5=y_1$, then $\deg(y_1)=6$, a contradiction. Hence, we have $z_5\notin\{x_1,y_1,z_3\}$.
 Similarly,  the i-edge $z_3z_4$ lies on the second triangle $z_3z_4z_6z_3$ and $z_5\ne z_6$. Otherwise, we have $\deg(z_4)=4$, a contradiction.
The second face incident with the $i$-edge $x_1z_2$ is the triangle  $x_1z_2x_2x_1$ and the second face incident with the $i$-edge $y_1z_3$ is the triangle  $y_1z_3y_2y_1$. 
The vertices $x_2$ and $y_2$ are the fifth neighbours of $x_1$ and $y_1$, respectively, and so they are e-vertices and the edges $x_1x_2$ and $y_1y_2$ are e-edges.
 The vertices $z_2$ and $z_3$ are i-vertices and $\deg(z_2)=\deg(z_3)=5$, therefore, the i-edges $z_2x_2$ and $z_3y_2$ lie on the triangles $z_2x_2z_5z_2$ and $z_3y_2z_6z_3$, respectively. Note that $x_2\ne y_2$, otherwise, we have $\deg(x_2)=6$, a contradiction. 
The vertex $z_5$ is an i-vertex, otherwise, $\deg(x_2)=3<k$, a contradiction. Hence, The edge $z_4z_5$ is an i-edge and so it lies on the second triangle $z_4z_5z_6z_4$. Now, the i-edge $z_5z_6$ lies on the second triangle $z_5z_6z_7z_5$. We have $z_7\notin\{x_2,y_2\}$, otherwise, if $z_7=x_2$ or $z_7=y_2$ then $\deg(z_5)=4$ or $\deg(z_6)=4$, respectively, a contradiction. Since $\deg(z_5)=\deg(z_6)=5$, the i-edges $x_2z_5$ and $y_2z_6$ lie on the triangles $x_2z_5z_7x_2$ and $y_2z_6z_7y_2$, respectively.
Note that $z_7$ is an i-vertex. Otherwise, $\deg(x_2)=4$, a contradiction.
We call the fifth neighbour of $z_7$ as $y$ and so the i-edges $x_2z_7$ and $y_2z_7$ lie on two triangles $x_2z_7yx_2$ and $y_2z_7yy_2$.
Since $y$ is the fifth neighbour of $x_2$ and $y_2$, the edges $y_2y$ and $x_2y$ are e-edges and so the graph is completed, but it has a vertex of degree 3, a contradiction.
\end{figwindow}

\begin{figwindow}[4,r,{
\begin{tikzpicture}[scale=1,rotate=180]
\draw (0,0)--(3,2)--(0,6)--(-3,2)--(0,0);
\draw (0,0)--(0,1)--(3,2)--(.75,3.7)--(0,4.3)--(-.75,3.7)--(-3,2)--(0,1);
\draw (.75,5)--(.75,2)--(0,1)--(-.75,2)--(-.75,5);
\draw (0,2.85)--(0,4.3)--(.75,5)--(-.75,5)--(0,4.3);
\draw (-3,2)--(3,2);
\draw (.75,3.7)--(-.75,2);\draw (-.75,3.7)--(.75,2);
 \filldraw (0,6) circle (2pt);\node at (0,6.3){$y$};
\filldraw (-.75,5) circle (2pt);\node at (-1.1,5){$y_2$};
\filldraw (-3,2) circle (2pt);\node at (-3.3,2){$y_1$};
\filldraw (.75,5) circle (2pt);\node at (1.1,5){$x_2$};
\filldraw (3,2) circle (2pt);\node at (3.3,2){$x_1$};
\filldraw (0,0) circle (2pt);\node at (0,-.3){$x$};
\filldraw (0,4.3) circle (2pt);\node at (0,4.6){$z_7$};  
\filldraw (-.75,3.7) circle (2pt);\node at (-1,3.8){$z_6$};  
\filldraw (.75,3.7) circle (2pt);\node at (1,3.8){$z_5$}; 
\filldraw (0,2.85) circle (2pt);\node at (0,2.5){$z_4$};
\filldraw (-.75,2) circle (2pt);\node at (-.75,1.7){$z_3$};  
\filldraw (.75,2) circle (2pt);\node at (.75,1.7){$z_2$};  
\filldraw (0,1) circle (2pt);\node at (0,1.3){$z_1$}; 
\end{tikzpicture}
},{$k=5, d=3, k_0=3$}]
\noindent
\textbf{Case 7.} If $k=5, d=3, k_0=3$, then $x$ has three neighbours. Let $N(x)=\{x_1,z_1,y_1\}$. By Theorem \ref{i-vertex}(iii), we choose the vertex $z_1$ as an i-vertex and other two neighbours of $x$ as e-vertices. The i-edge $xz_1$ lies on the two triangles $xz_1x_1x$ and $xz_1y_1x$. The edge $x_1z_1$ is an i-edges. The second face consisting the edge $x_1z_1$ is the triangle $x_1z_1z_2x_1$. 
Note that $z_2\ne y_1$, otherwise, $\deg(z_1)=3$, a contradiction.
Similarly, the edge $z_1y_1$ is an i-edges and the second triangle consisting the edge $z_1y_1$ is the triangle  $z_1y_1z_3z_1$. 
The vertices $z_2$ and $z_3$ are i-vertices. Otherwise, $\deg(x_1)=3<k$ or $\deg(y_1)=3<k$, a contradiction. 
Also, we have $z_2\ne z_3$, otherwise, $\deg(z_1)=4$, a contradiction.
The second triangle incident with the i-edge $z_1z_2$ is the triangle  $z_1z_2z_3z_1$ and the second triangle incident with the i-edge $z_2z_3$ is the triangle  $z_2z_3z_4z_2$. 
Note that $z_4\notin\{x_1,y_1\}$, otherwise, $\deg(z_2)=3$ or $\deg(z_3)=3$, a contradiction. 
The i-edge $x_1z_2$ lies on the second triangle $x_1z_2z_5x_1$. If $z_5=z_4$, then $\deg(z_2)=4$, a contradiction, and so we have $z_5\ne z_4$ and the i-edge $z_2z_5$ lies on the second triangle $z_2z_5z_4z_2$.
 Note that the vertex $z_5$ is the fifth neighbour of $z_2$ and $z_5\ne y_1$, otherwise, $\deg(y_1)=6$, a contradiction.
  Similarly, the i-edge $y_1z_3$ lies on the second triangle $y_1z_3z_6y_1$ and the i-edge $z_3z_6$ lies on the second triangle $z_3z_4z_6z_3$. 
 We know that $z_5$ and $z_6$ are i-vertices, otherwise, $\deg(x_1)=4$ or $\deg(y_1)=4$, a contradiction. 
The i-edge $z_4z_5$ lies on the second triangle $z_4z_5z_7z_4$. Note that $z_6\ne z_7$, otherwise, $\deg(z_4)=4$, a contradiction.
 Now, the i-edge $z_4z_6$ lies on the second triangle $z_4z_6z_7z_4$. 
 The vertices $z_5$ and $z_6$ are i-vertices. Otherwise, $\deg(x_1)=4$ or $\deg(y_1)=4$, a contradiction.
The i-edge $x_1z_5$ lies on the second triangle $x_1z_5x_2x_1$ and  the second triangle consisting the i-edge $x_2z_5$ is the triangle  $x_2z_5z_7x_2$. 
Similarly, the i-edge $y_1z_6$ lies on the second triangle $y_1z_6y_2y_1$ and  the second triangle consisting the i-edge $y_2z_6$ is the triangle  $y_2z_6z_7y_2$.
 If $x_2=y_2$, then $\deg(z_7)=4$, a contradiction, and so $x_2\ne y_2$. Since the vertices $x_2$ and $y_2$ are the fifth neighbours of $x_1$ and $y_1$, respectively, they are e-vertices and $x_1x_2$ and $y_1y_2$ are e-edges.
 The edge $x_2z_7$ is an i-edge, otherwise, $\deg(x_2)=3<k$, a contradiction.  
 Since the vertex $z_7$ is already of degree 5, the second triangle consisting $x_2z_7$ is $x_2z_7y_2x_2$. Now, the edge $x_2y_2$ is an i-edge, otherwise, $\deg(x_2)=4<k$, a contradiction. The second triangle consisting $x_2y_2$ is $x_2y_2yx_2$. The vertex $y$ is the fifth neighbour of $x_2$ and $y_2$ and so it is an e-vertex and the edges $x_2y$ and $y_2y$ are e-edges. The graph is completed but it has a vertex of degree 2, a contradiction.
\end{figwindow}

\begin{figwindow}[12,r,{
\begin{tikzpicture}[scale=1.2]
\draw (0,2)--(2,0)--(0,-2)--(-2,0)--(0,2);
\draw (0,2)--(.5,1)--(1,0)--(.5,-1)--(0,-2)--(-1,0)--(0,2);
\draw (-2,0)--(-.5,1)--(.5,1)--(2,0)--(.5,-1)--(-.5,-1)--(-2,0);
\draw (0,.5)--(0,-.5);
\draw (-2,0)--(-1,0);\draw (2,0)--(1,0);
 \draw (-.5,-1)--(0,-.5)--(1,0)--(0,.5)--(-.5,1);
 \draw (.5,-1)--(0,-.5)--(-1,0)--(0,.5)--(.5,1);
\filldraw (0,2) circle (2pt);\node at (0,2.3){$x$};
\filldraw (-2,0) circle (2pt);\node at (-2.3,0){$x_1$};
\filldraw (0,-2) circle (2pt);\node at (0,-2.3){$y$};
\filldraw (2,0) circle (2pt);\node at (2.3,0){$y_1$};
\filldraw (-.5,1) circle (2pt);\node at (-.7,1.1){$z_1$};  
\filldraw (.5,1) circle (2pt);\node at (.7,1.1){$z_2$};  
\filldraw (0,.5) circle (2pt);\node at (0,.7){$z_3$};  
\filldraw (-1,0) circle (2pt);\node at (-.7,0){$z_4$}; 
\filldraw (1,0) circle (2pt);\node at (.7,0){$z_5$}; 
\filldraw (0,-.5) circle (2pt);\node at (0,-.7){$z_6$};
 \filldraw (-.5,-1) circle (2pt);\node at (-.7,-1.1){$z_7$};
\filldraw (.5,-1) circle (2pt);\node at (.7,-1.1){$z_8$};  
\end{tikzpicture}
},{$k=5, d=3, k_0=4$}]
\noindent
\textbf{Case 8.} If $k=5, d=3, k_0=4$, then $x$ has four neighbours. Let $N(x)=\{x_1,z_1,z_2,y_1\}$. By Theorem \ref{i-vertex}(iii), we choose the vertices $z_1$ and $z_2$ as  i-vertices and other two neighbours of $x$ as e-vertices. The e-edge $xx_1$ lies on a triangular face $xx_1wx$. If $w=y_1$ then we have $\deg(x)=2$, a contradiction. Hence, $w\in\{z_1,z_2\}$. We say that $xz_1x_1x$ is the triangular face incident to $xx_1$. Similarly, the e-edge $yy_1$ lies on a triangular face $yy_1wx$. If $w=z_1$ then we have $\deg(x)=3$, a contradiction. Hence, $w=z_2$ and $yz_2y_1y$ is the triangular face incident to $yy_1$. 
The triangle $xz_1z_2x$ is the second triangle incident to the i-edge $xz_1$. The edge $z_1z_2$ is an i-edge and lies on the second triangle $z_1z_2z_3z_1$. We have $z_3\notin\{x_1,y_1\}$. For example, if $z_3=x_1$, then $\deg(z_1)=3<k$, a contradiction.
The i-edge $x_1z_1$ lies on the second triangle $x_1z_1z_4x_1$ and $z_4\ne z_3$, otherwise, $\deg(z_1)=4$, a contradiction. 
Since $z_4$ is the fifth neighbour of $z_1$, the triangle $z_1z_3z_4z_1$ is the second triangle incident to the i-edge $z_1z_3$. Similarly, 
the i-edge $y_1z_2$ lies on the second triangle $y_1z_2z_5y_1$ and $z_5\ne z_3$. Also, $z_2z_3z_5z_2$ is the second triangle incident with the i-edge $z_2z_3$.
 If $z_4=z_5$, then $\deg(z_3)=3$, a contradiction and if $z_4=y_1$ or $z_5=x_1$, then $\deg(y_1)=6$ or $\deg(x_1)=6$, respectively, a contradiction. The vertices $z_4$ and $z_5$ are i-vertices. Otherwise, $\deg(x_1)=3$ or $\deg(y_1)=3$, a contradiction. The i-edge $z_4z_3$ lies on the second triangle $z_3z_4z_6z_3$. 
We have $z_6\ne z_5$. Otherwise, $\deg(z_3)=4$, a contradiction. $z_6$  is the fifth neighbour of $z_3$ and so the i-edge $z_3z_5$ lies on the second triangle $z_3z_5z_6z_3$.
Now we notice that $z_6\notin\{x_1,y_1\}$. Because, if $z_6=x_1$, then $\deg(z_4)=3$, a contradiction and  if $z_6=y_1$, then $\deg(z_5)=3$, a contradiction. 
The i-edge $x_1z_4$ lies on the second triangle $x_1z_4z_7x_1$ and $z_7\ne z_6$. Otherwise, $\deg(z_4)=4$, a contradiction. Now, the triangle $z_4z_6z_7z_4$ is the second triangle incident to the i-edge $z_4z_7$. The vertex $z_7$ is an i-vertex. Otherwise, we have $\deg(x_1)=4$, a contradiction. 
Similarly, the i-edge $y_1z_5$ lies on the second triangle $y_1z_5z_8y_1$ and the triangle $z_5z_6z_8z_5$ is the second triangle incident to the i-edge $z_5z_8$. 
Also, the vertex $z_8$ is an i-vertex. Furthermore, we have $z_7\ne z_8$, otherwise, $\deg(z_6)=4$, a contradiction.
Since we know all five neighbours of $z_6$, the i-edge $z_6z_7$ lies on the second triangle $z_6z_7z_8z_6$. Now, the i-edge $z_7z_8$ lies on the second triangle $z_7z_8yz_7$. We have $y\notin\{x_1,y_1\}$. Otherwise, $\deg(z_7)=4$ or $\deg(z_8)=4$, a contradiction. Finally, The second triangles incident to the i- edges $x_1z_7$ and $y_1z_8$ are $x_1z_7yx_1$ and $y_1z_8yy_1$, respectively. The vertex $y$ is the fifth neighbour of both e-vertices $x_1$ and $y_1$, Hence, $y$ is an e-vertex and $x_1y$ and $y_1y$ are e-edges. Now, the graph is completed but it has two vertices of degree 4, a contradiction.
\end{figwindow} 
\end{proof}

\begin{lemma}\label{d1<6} If a $(k;d_1^{f-1}d_2^{1})$-graph be a $k$-regular $1$-NPG, then $3\le d_1\le5$ for $k=3$ and $d_1=3$ for $k=4,5$.
\end{lemma}
\begin{proof}
The graph is $k$-regular and so we have $kn=2m$. By Proposition \ref{2e} we have $2m=d_1(f-1)+d_2$ and by Euler's formula $f-1=m-n+1=\frac{1}{2}(k-2)n+1$. Therefore, $kn=\frac{1}{2}(k-2)nd_1+d_1+d_2$ or $d_1=\dfrac{2kn}{(k-2)n+2}-\dfrac{2d_2}{(k-2)n+2}$ which implies that $d_1<\frac{2k}{k-2}$. Now, if $k=3$, then $3\le d_1\le5$ and if $k\in\{4,5\}$, then $ d_1=3$.
\end{proof}

In \cite{Diestel}, it is proved that:    

\begin{lemma} \cite[Lemma 4.2.1]{Diestel} Let $G$ be a plane graph, $F$ a face, and $H$ a subgraph of  $G$.
\begin{itemize}
  \setlength\itemsep{.01em}
\item[(i)]
 $H$ has a face $F'$ containing $F$.
\item[(i)]
 If the frontier of $F$ lies in $H$, then $F'=F$.
 \end{itemize}
\end{lemma}

\begin{observation}\label{outerface} If $G$ be a plane graph with the outerface $F_G$ and $H$ be a subgraph of $G$ with the outerface $F_H$, then $F_G\subseteq F_H$. Furthermore, each vertex $u\in V(H)\cap\partial(F_G)$ belongs to the $\partial(F_H)$.
\end{observation}

\begin{lemma}\label{2.13} Let $G$ be a connected plane graph, $x$ a cut-vertex in $G$, and $H$ a component of $G-x$. If $G_1=\langle\{x\}\cup V(H)\rangle$ and $G_2=\langle V(G)\setminus V(H)\rangle$, then:
\begin{itemize}
  \setlength\itemsep{.01em}
\item[(i)]
 $G_1$ has a face containing $G_2\setminus\{x\}$.
\item[(ii)]
 $G_2$ has a face containing $G_1\setminus\{x\}$.
\item[(iii)]
 $G$ has a face incident to $x$ at least twice.
\end{itemize}
\end{lemma}
\begin{proof}
(i) Two subgraphs $G_1$ and $G_2$ are connected and they have only one vertex, $x$, in common. For a vertex $x_1(\ne x)$ in $G_1$ there exists a face $F_2$ in $G_2$ such that $x_1\in F_2$. Now consider another vertex $x'_1(\ne x)$ in $G_1$. There exists a path between $x_1$ and $x'_1$ in $G_1$ independent from $x$. This path is induced from $G$ and does not meet the graph $G_2$ and so $x'_1\in F_2$. Hence, all vertices of $G_1$ are in the the face $F_2$ except the vertex $x$ that lies on the boundary of $F_2$ with the boundary walk $xy_1y_2\cdots y_{s_2}x$.\\
(ii) Similarly, for a vertex $x_2\ne x$ in $G_2$ there exists a face $F_1$ in $G_1$ such that $x_2\in F_1$. Now consider another vertex $x'_2\ne x$ in $G_2$. There exists a path between $x_2$ and $x'_2$ in $G_2$ independent from $x$. This path is induced from $G$ and does not meet the graph $G_1$ and so $x'_2\in F_1$. Hence, all vertices of $G_2$ are in the face $F_1$ except the vertex $x$ that lies on the boundary of $F_1$ with the boundary walk $xz_1z_2\cdots z_{s_1}x$. \\
(iii) Now, $F=F_1\cap F_2$ is a region in the plane and $F\cap G=\emptyset$. Therefore, $F$ is a face of $G$ with the boundary walk $xy_1y_2\cdots y_{s_2}xz_1z_2\cdots z_{s_1}x$ such that $y_i\ne z_j$ for all $i$ and $j$. 
\end{proof}

\begin{corollary}\label{cor345}
Let $G$ be a connected planar graph and $x$ be a cut-vertex of $G$.
\begin{itemize}
  \setlength\itemsep{.01em}
\item[(i)]
 $x$ lies on a face with the length at least $4$.
 \item[(ii)]
 If the length of each face incident to $x$ is less than $6$, then $x$  has a neighbour with degree $1$.
\end{itemize}
\end{corollary}

\begin{proof} By Lemma \ref{2.13}, $G$ has a face $F$ with the boundary walk $xy_1y_2\cdots y_{s_2}xz_1z_2\cdots z_{s_1}x$ such that $y_i\ne z_j$ for all $i$ and $j$. Thus, the length of this walk is equal to $s_1+s_2+2$. We know  $s_1,s_2\ge1$ that implies (i). If each face incident to $x$ has the length less than 6, then $s_1+s_2+2\le 5$ and so we have three cases: $s_1=s_2=1$, $s_1=1,s_2=2$ and $s_1=2,s_2=1$. In the first, we deduce that $G=P_3$ with two endvertices adjacent to $x$.
In the second and third, $z_1$ or $y_1$, respectively, is an endvertex adjacent to $x$.
\end{proof}

\begin{lemma} Let $G$ be a $1$-NPG, each cut-vetrtex in $G$ lies on the exceptional face of $G$.
\end{lemma}

\begin{proof}
Since $G$ is a $1$-NPG, $G$ has no endvertex and so by Corollary \ref{cor345}, each cut-vertex lies on a face of the length at least $6$ and by Lemma \ref{d1<6}, $x$ lies on the exceptional face of $G$.
\end{proof}

\begin{corollary}\label{cor-1npg} Let $G$ be a $1$-NPG and $x$ is a cut-vertex. The length of the unique exceptional face is at least $6$ and $x$ lies on it.
\end{corollary}

\begin{definition}\cite[Definition 4.1.16]{West-2001}
A block of a graph $G$ is a maximal connected subgraph of $G$ that has no cut-vertex. If $G$ itself is connected and has no cut-vertex, then $G$ is a block. 
\end{definition}
\begin{remark}\cite[Pages 155,156]{West-2001}\label{rem0}
The blocks of a graph are its isolated vertices, its cut-edges, and its maximal 2-connected subgraphs.
Two blocks in a graph share at most one vertex, it must be a cut-vertex and a cut-vertex is belong to at least two blocks. 
A graph that is not a single block has at least two blocks (leaf blocks) that each contain exactly one cut-vertex of the graph.
\end{remark}

\begin{theorem} There is no finite, planar, regular, connected graph that has all but one face of the same degree and a single face of a different degree.
\end{theorem}

\begin{proof} Assume by the contrary, $G$ is a 1-NPG
with at least one cut-vertex. By \cite[Lemma 4]{Keith-Froncek-Kreher-1}, we have $3\le \deg(w)\le 5$ for all $w\in V(G)$.
We consider an embedding of $G$ such that the exceptional face of $G$ be the outerface of $G$.
Since $G$ is not a single block, by Remark \ref{rem0}, $G$ has two leaf blocks that each contain exactly one cut-vertex of $G$. We consider a leaf block $B$, containing the vertex $x$ as a cut-vertex of $G$. Since $G$ is connected, $B$ is not an isolated vertex and $x$ has at least a neighbour $y$ in $B$ with $\deg_B(y)=\deg_G(y)\ge 3$ and so $B$ has at least four vertices and $B$ is not a cut-edge of $G$. Therefore, $B$ is a maximal 2-connected subgraph of $G$. Immediately, we have $\deg_B(x)\ge 2$ and since $x$ has at least one neighbour in another block we see that $\deg_B(x)\le k-1$. On the other hand, for each vertex $z\in V(B)\setminus\{x\}$ we have $\deg_B(z)=\deg_G(x)=k$ and $2\le\deg_B(x)\le k-1$. 
By Corollary \ref{cor-1npg}, $x$ lies on the boundary of  the unique exceptional face (=outerface) of $G$ with the length at least 6.  
By Lemma \ref{2.13}, this face incident to $x$ at least twice.
 Let $xy_1y_2\cdots y_{s_2}xz_1z_2\cdots z_{s_1}x$ be the boundary walk of the outerface of $G$ such that $y_i\ne z_j$ for all $i$ and $j$.

 Here, the block $B$ plays the role of $G_1$ in Lemma \ref{2.13} \hfill and \hfill so \hfill a \hfill part \hfill of \hfill the \hfill boundary \hfill walk\\ $xy_1y_2\cdots y_{s_2}xz_1z_2\cdots z_{s_1}x$ , say $xy_1y_2\cdots y_{s_2}x$, is the boundary of a face of $B$. Now, by Observation \ref{outerface}, $xy_1y_2\cdots y_{s_2}x$ is the boundary walk of the outerface of $B$. Indeed, by 2-connectivity of $B$, $xy_1y_2\cdots y_{s_2}x$ is a cycle.
 
 Since the internal faces of the block $B$ are the internal faces of $G$, they have the same lengths. Therefore, $B$ is a connected planar graph such that all its internal faces have the same lengths, all its vertices, except $x$ lying on the outerface, have the same degree $k$ and $2\le \deg_B(x)\le k-1$. This contradicts the Theorem \ref{th3.8}. Consequently, there is no 1-NPG.
\end{proof}



\end{document}